\documentclass[a4paper,twoside]{article}
\usepackage{a4}
\usepackage{amssymb}
\usepackage{amsmath}
\usepackage{upref}
\usepackage{url}
\usepackage[active]{srcltx}
\usepackage[pagebackref,colorlinks,citecolor=blue,linkcolor=blue,urlcolor=blue]{hyperref}
\usepackage[dvipsnames]{color}
\usepackage{mathtools}
\allowdisplaybreaks[2] 
\newcount\minutes \newcount\hours
\hours=\time
\divide\hours 60
\minutes=\hours
\multiply\minutes -60
\advance\minutes \time
\newcommand{\klockan}{\the\hours:{\ifnum\minutes<10 0\fi}\the\minutes}
\newcommand{\tid}{\today\ \klockan}
\newcommand{\prtid}{\smash{\raise 10mm \hbox{\LaTeX ed \tid}}}
\renewcommand{\prtid}{}
\makeatletter
\def\sectionmark#1{} 
\def\subsectionmark#1{}
\newcommand{\sectnr}{\ifnum \c@secnumdepth >\z@
                 \thesection.\hskip 1em\relax \fi}
\def\@evenhead{\footnotesize\rm\thepage\hfil\leftmark\hfil\llap{\prtid}}
\def\@oddhead{\footnotesize\rm\rlap{\prtid}\hfil\rightmark\hfil\thepage}
\def\tableofcontents{\section*{Contents} 
 \@starttoc{toc}}
\makeatother
\makeatletter
\def\@biblabel#1{#1.}
\makeatother
\makeatletter
\let\Thebibliography=\thebibliography
\renewcommand{\thebibliography}[1]{\def\@mkboth##1##2{}\Thebibliography{#1}
\addcontentsline{toc}{section}{References}
\frenchspacing 
\setlength{\@topsep}{0pt}
\setlength{\itemsep}{0pt}%
\setlength{\parskip}{0pt plus 2pt}%
}
\makeatother
\makeatletter
\def\mdots@{\mathinner.\nonscript\!.%
 \ifx\next,.\else\ifx\next;.\else\ifx\next..\else
 \nonscript\!\mathinner.\fi\fi\fi}
\let\ldots\mdots@
\let\cdots\mdots@
\let\dotso\mdots@
\let\dotsb\mdots@
\let\dotsm\mdots@
\let\dotsc\mdots@
\def\vdots{\vbox{\baselineskip2.8\p@ \lineskiplimit\z@
    \kern6\p@\hbox{.}\hbox{.}\hbox{.}\kern3\p@}}
\def\ddots{\mathinner{\mkern1mu\raise8.6\p@\vbox{\kern7\p@\hbox{.}}%
    \raise5.8\p@\hbox{.}\raise3\p@\hbox{.}\mkern1mu}}
\makeatother
\makeatletter
\let\Enumerate=\enumerate
\renewcommand{\enumerate}{\Enumerate%
\setlength{\@topsep}{0pt}
\setlength{\itemsep}{0pt}%
\setlength{\parskip}{0pt plus 1pt}%
\renewcommand{\theenumi}{\textup{(\alph{enumi})}}%
\renewcommand{\labelenumi}{\theenumi}%
}
\let\endEnumerate=\endenumerate
\renewcommand{\endenumerate}{\endEnumerate\unskip}
\makeatother
\makeatletter
\def\@seccntformat#1{\csname the#1\endcsname.\quad}
\makeatother
\makeatletter
\long\def\@makecaption#1#2{%
  \vskip\abovecaptionskip
  \sbox\@tempboxa{ #1. #2}%
  \ifdim \wd\@tempboxa >\hsize
    #1. #2\par
  \else
    \global \@minipagefalse
    \hb@xt@\hsize{\hfil\box\@tempboxa\hfil}%
  \fi
  \vskip\belowcaptionskip}
\makeatother
\makeatletter
\renewcommand\section{\@startsection {section}{1}{\z@}%
                                   {-3.5ex \@plus -1ex \@minus -.2ex}%
                                   {2.3ex \@plus.2ex}%
                                   {\normalfont\Large\bfseries\boldmath}}
\renewcommand\subsection{\@startsection{subsection}{2}{\z@}%
                                     {-3.25ex\@plus -1ex \@minus -.2ex}%
                                     {1.5ex \@plus .2ex}%
                                     {\normalfont\large\bfseries\boldmath}}
\renewcommand\subsubsection{\@startsection{subsubsection}{3}{\z@}%
                                     {-3.25ex\@plus -1ex \@minus -.2ex}%
                                     {1.5ex \@plus .2ex}%
                                     {\normalfont\normalsize\bfseries\boldmath}}
\renewcommand\paragraph{\@startsection{paragraph}{4}{\z@}%
                                    {3.25ex \@plus1ex \@minus.2ex}%
                                    {-1em}%
                                    {\normalfont\normalsize\bfseries\boldmath}}
\renewcommand\subparagraph{\@startsection{subparagraph}{5}{\parindent}%
                                       {3.25ex \@plus1ex \@minus .2ex}%
                                       {-1em}%
                                      {\normalfont\normalsize\bfseries\boldmath}}
\makeatother
\newcommand{\authortitle}[2]{\author{#1}\title{#2}\markboth{#1}{#2}}
\newcommand{\art}[6]{{\sc #1, \rm #2, \it #3 \bf #4 \rm (#5), \mbox{#6}.}}
\newcommand{\book}[3]{{\sc #1, \it #2, \rm #3.}}
\newcommand{\auth}[2]{{#1, #2.}}
\newcommand{\AND}{{\rm and }}
\RequirePackage{amsthm}
\newtheoremstyle{descriptive}%
  {\topsep}   
  {\topsep}   
  {\rmfamily} 
  {}          
  {\bfseries} 
  {.}         
  { }         
  {}          
\newtheoremstyle{propositional}%
  {\topsep}   
  {\topsep}   
  {\itshape}  
  {}          
  {\bfseries} 
  {.}         
  { }         
  {}          
\newtheoremstyle{remarkstyle}%
  {\topsep}   
  {\topsep}   
  {\rmfamily}  
  {}          
  {\itshape} 
  {.}         
  { }         
  {}          
\theoremstyle{propositional}
\newtheorem{thm}{Theorem}[section]
\newtheorem{lem}[thm]{Lemma}

\theoremstyle{descriptive}

\newtheorem{deff}[thm]{Definition}
\theoremstyle{remarkstyle}
\newtheorem{remark}[thm]{Remark}
\newtheorem{example}[thm]{Example}
\makeatletter
\renewenvironment{proof}[1][\proofname]{\par
  \pushQED{\qed}%
  \normalfont
  \trivlist
  \item[\hskip\labelsep
        \itshape
    #1\@addpunct{.}]\ignorespaces
}{%
  \popQED\endtrivlist\@endpefalse
}
\makeatother
\newcommand{\setm}{\setminus}
\newcommand{\strutdepth}{\hbox{\vrule height0pt depth4.5pt width0pt}}
\newcommand{\strutheight}{\hbox{\vrule height10pt depth0pt width0pt}}
{\catcode`p =12 \catcode`t =12 \gdef\eeaa#1pt{#1}}      
\def\accentadjtext#1{\setbox0\hbox{$#1$}\kern   
                \expandafter\eeaa\the\fontdimen1\textfont1 \ht0 }
\def\accentadjscript#1{\setbox0\hbox{$#1$}\kern 
                \expandafter\eeaa\the\fontdimen1\scriptfont1 \ht0 }
\def\accentadjscriptscript#1{\setbox0\hbox{$#1$}\kern   
                \expandafter\eeaa\the\fontdimen1\scriptscriptfont1 \ht0 }
\def\accentadjtextback#1{\setbox0\hbox{$#1$}\kern       
                -\expandafter\eeaa\the\fontdimen1\textfont1 \ht0 }
\def\accentadjscriptback#1{\setbox0\hbox{$#1$}\kern     
                -\expandafter\eeaa\the\fontdimen1\scriptfont1 \ht0 }
\def\accentadjscriptscriptback#1{\setbox0\hbox{$#1$}\kern 
                -\expandafter\eeaa\the\fontdimen1\scriptscriptfont1 \ht0 }
\def\itoverline#1{{\mathsurround0pt\mathchoice
        {\rlap{$\accentadjtext{\displaystyle #1}
                \accentadjtext{\vrule height1.593pt}
                \overline{\phantom{\displaystyle #1}
                \accentadjtextback{\displaystyle #1}}$}{#1}}
        {\rlap{$\accentadjtext{\textstyle #1}
                \accentadjtext{\vrule height1.593pt}
                \overline{\phantom{\textstyle #1}
                \accentadjtextback{\textstyle #1}}$}{#1}}
        {\rlap{$\accentadjscript{\scriptstyle #1}
                \accentadjscript{\vrule height1.593pt}
                \overline{\phantom{\scriptstyle #1}
                \accentadjscriptback{\scriptstyle #1}}$}{#1}}
        {\rlap{$\accentadjscriptscript{\scriptscriptstyle #1}
                \accentadjscriptscript{\vrule height1.593pt}
                \overline{\phantom{\scriptscriptstyle #1}
                \accentadjscriptscriptback{\scriptscriptstyle #1}}$}{#1}}}}
\def\vint{\mathop{\mathchoice%
          {\setbox0\hbox{$\displaystyle\intop$}\kern 0.22\wd0%
           \vcenter{\hrule width 0.6\wd0}\kern -0.82\wd0}%
          {\setbox0\hbox{$\textstyle\intop$}\kern 0.2\wd0%
           \vcenter{\hrule width 0.6\wd0}\kern -0.8\wd0}%
          {\setbox0\hbox{$\scriptstyle\intop$}\kern 0.2\wd0%
           \vcenter{\hrule width 0.6\wd0}\kern -0.8\wd0}%
          {\setbox0\hbox{$\scriptscriptstyle\intop$}\kern 0.2\wd0%
           \vcenter{\hrule width 0.6\wd0}\kern -0.8\wd0}}%
          \mathopen{}\int}
\newcommand{\grad}{\nabla}

\DeclareMathOperator{\Div}{div}

\DeclareMathOperator{\Lip}{Lip}
\newcommand{\Lipc}{{\Lip_c}}

\newcommand{\bdry}{\partial}
\newcommand{\bdy}{\bdry}
\newcommand{\loc}{_{\rm loc}}
\newcommand{\limplus}{{\mathchoice{\raise.17ex\hbox{$\scriptstyle +$}}
                {\raise.17ex\hbox{$\scriptstyle +$}}
                {\raise.1ex\hbox{$\scriptscriptstyle +$}}
                {\scriptscriptstyle +}}}
\newcommand{\alp}{\alpha}

\newcommand{\al}{\alpha}

\newcommand{\ga}{\gamma}
\newcommand{\Om}{\Omega}

\renewcommand{\phi}{\varphi}
\newcommand{\eps}{\varepsilon}
\newcommand{\p}{{$p\mspace{1mu}$}}
\newcommand{\R}{\mathbf{R}}
\newcommand{\Sphere}{\mathbf{S}}

\newcommand{\Qaln}{Q_{\al,n}}

\newcommand{\Wp}{W^{1,p}}

\newcommand{\Wploc}{W^{1,p}\loc}
\newcommand{\phit}{{\widetilde{\phi}}}
\newcommand{\phib}{{\overline{\phi}}}
\newcommand{\etat}{{\tilde{\eta}}}
\newcommand{\Ih}{\hat{I}}
\newcommand{\coloneqq}{:=}
\newcommand{\Qt}{\widetilde{Q}}
\newcommand{\Qhat}{\widehat{Q}}
\newcommand{\bQ}{\itoverline{Q}}
\newcommand{\ub}{\bar{u}}
\newcommand{\stab}[2]{\noindent\begin{tabular}{@{}r@{}}#1\\ #2\end{tabular}}
\numberwithin{equation}{section}
\newenvironment{ack}{\medskip{\it Acknowledgement.}}{}
\begin{document}
\authortitle{Anders Bj\"orn, Jana Bj\"orn
    and Ismail Mirumbe}
{The quasisuperminimizing constant for the minimum
of two quasisuperminimizers in $\R^n$}

\author{
Anders Bj\"orn \\
\it\small Department of Mathematics, Link\"oping University, \\
\it\small SE-581 83 Link\"oping, Sweden\/{\rm ;}
\it \small anders.bjorn@liu.se
\\
\\
Jana Bj\"orn \\
\it\small Department of Mathematics, Link\"oping University, \\
\it\small SE-581 83 Link\"oping, Sweden\/{\rm ;}
\it \small jana.bjorn@liu.se
\\
\\
Ismail Mirumbe\\
\it\small Department of Mathematics, Makerere University, \\
\it\small P.O. Box 7062, Kampala, Uganda\/{\rm ;}
\it\small mirumbe@cns.mak.ac.ug
\\
}

\date{Preliminary version, \today}
\date{}

\maketitle

\noindent{\small
 {\bf Abstract}.
It was shown in
Bj\"orn--Bj\"orn--Korte (``\emph{Minima of quasisuperminimizers}'',
Nonlinear Anal. {\bf 155} (2017), 264--284)
that $u:=\min\{u_1,u_2\}$ is a $\bQ$-quasisuperminimizer
if $u_1$ and $u_2$ are $Q$-quasisuperminimizers and $\bQ=2Q^2/(Q+1)$.
Moreover, one-dimensional examples 
therein show  
that $\bQ$ is close to optimal.
In this paper we give similar examples in higher dimensions.
The case when $u_1$ and $u_2$ have different quasisuperminimizing
constants is considered as well.
}

\bigskip


\medskip


\section{Introduction}

Let $\Omega \subset \R^n$ be a nonempty open set and $1< p < \infty$. 
A function $u\in \Wploc(\Omega)$ is a 
\emph{$Q$-quasi\/\textup{(}super\/\textup{)}minimizer} in
$\Omega$, with $Q\ge1$, if
\begin{equation}\label{eqmain}
   \int_{\varphi \neq 0} |\nabla u|^{p}\,dx
    \leq      Q\int_{\varphi \neq 0} |\nabla( u + \varphi)|^{p}\,dx
\end{equation}
for all (nonnegative) $\varphi \in  \Wp_0(\Omega)$.  
Quasiminimizers were introduced
by Giaquinta--Giusti~\cite{GG1} as a tool for a unified treatment
of variational integrals, elliptic equations and systems, and
quasiregular mappings on $\R^n$. 

Quasi(super)minimizers have an interesting theory already
in the one-dimensional case, see e.g.\ \cite{GG1}
and Martio--Sbordone~\cite{MaSb}.
Kinnunen--Martio~\cite{KiMa03} showed
that one can build a rich potential theory
based on quasiminimizers.
In particular, they introduced quasisuperharmonic functions, which
are related to quasisuperminimizers in a similar way as
superharmonic functions are related to supersolutions.
See 
Bj\"orn--Bj\"orn--Korte~\cite{BBKorte}
for further references.

Kinnunen--Martio~\cite[Lemmas~3.6 and~3.7]{KiMa03} also showed
that if 
$u_j$ is a $Q_j$-quasisuperminimizer
in $\Om \subset \R^n$ (or in a metric space), $j=1,2$,
then 
$u:=\min\{u_1,u_2\}$ is a 
$\min\{Q_1+Q_2,Q_1Q_2\}$-quasisuperminimizer.
Bj\"orn--Bj\"orn--Korte~\cite{BBKorte} improved upon this result
in the following way.

\begin{thm} \label{thm-BBK}
\textup{(Theorem~1.2 in \cite{BBKorte})}
Let $u_j$ be a $Q_j$-quasisuperminimizer
in $\Om \subset \R^n$ \textup{(}or in a metric space\/\textup{)}, $j=1,2$.
Then $u:=\min\{u_1,u_2\}$ is a $\bQ$-quasisuperminimizer
in $\Om$, where
\begin{equation}   \label{eq-def-q-for-min}
\bQ=\begin{cases}
  1, & \text{if } Q_1=Q_2=1, \\
  \displaystyle (Q_{1}+Q_{2}-2)\frac{Q_{1}Q_{2}}{Q_{1}Q_{2}-1}, & \text{otherwise.}
  \end{cases}
\end{equation}
In particular, if $Q_1=Q_2$, then $\bQ=2Q_1^2/(Q_1+1)$.
\end{thm}

It is not known whether $\bQ$ is optimal, but
it is the best upper bound known.
On the other hand, that $u$ is 
(in general) 
not better than a $\max\{Q_1,Q_2\}$-quasi\-super\-minimizer is rather obvious.

The first
examples
(and so far the only ones)
showing  that $u$ is (in general) \emph{not} a 
$\max\{Q_1,Q_2\}$-quasi\-super\-minimizer
were given in \cite{BBKorte}.
More precisely, if $1<Q_1 \le Q_2$ then there 
are $Q_j$-quasisuperminimizers $u_j$ on $(0,1) \subset \R$ such that
$u:=\min\{u_1,u_2\}$ is not a $Q_2$-quasisuperminimizer.
Estimates and concrete examples, showing 
that the constant $\bQ$ above is not too far from being optimal,
  were also given in~\cite{BBKorte}.
All examples therein
were on $(0,1)\subset\R$
and our aim in this paper is to obtain similar examples in higher
dimensions, i.e.\ for subsets of $\R^n$, $n \ge 2$.

The examples in \cite{BBKorte} (giving the best lower bounds)
were based on power functions $x \mapsto x^\alpha$ and reflections
of such functions.
For such functions, also in the higher-dimensional case on  $\R^n$, $n\ge2$,
optimal quasi(super)minimizing constants $Q(\alpha,p,n)$ were obtained in 
Bj\"orn--Bj\"orn~\cite{BBpower}, and these formulas for $Q(\alpha,p,n)$
(with $n=1$)
were used in the calculations in \cite{BBKorte}.

As power-type functions $x \mapsto |x|^\alpha$ only 
have point singularities, it seems difficult
to use them for higher-dimensional analogues of the examples
in \cite{BBKorte}.
Instead we base our examples
on log-power functions $\log^\alp |x|$
and $(-\log |x|)^\alp$.
Since 
$\log|ex|=1-(-\log|x|)$,
we
 are able to scale and translate
them and create higher-dimensional
examples on annuli, in the spirit of \cite{BBKorte}.
For this to be possible we need the log-powers to be quasisuperminimizers
which requires $p$ to be equal to the conformal dimension $n$.
In particular we obtain the following result.

\begin{thm} \label{thm0}
 Let $p=n \ge 2$ and $1 < Q_1 \leq Q_2$ be given. 
Then there are  functions
 $u_1$ and $u_2$ on $A:=\{x \in \R^n : 1/e < |x| < 1\}$ 
such that $u_j$ is a $Q_j$-quasisuperminimizer in $A$, $j = 1,2$,
 but $\min\{u_1, u_2\}$ is
 not a $Q_2$-quasiminimizer in~$A$.
\end{thm}

As in \cite{BBKorte} we also give lower bounds for the 
increase in the quasisuperminimizing constant and show
that these lower bounds are the same as in the 1-dimensional
case considered in~\cite{BBKorte}, see Section~\ref{sect-min}.
In Section~\ref{sect-dummy} we show that one
can add dummy variables to these examples, as well
as to those in \cite{BBKorte};
this is nontrivial and partly relies on
results from Bj\"orn--Bj\"orn~\cite{BBtensor}.

\begin{ack}
A.B. and J.B. were supported by the Swedish Research Council
grants 2016-03424 and 621-2014-3974, respectively,
while I.M.
was supported by the SIDA (Swedish International Development
Cooperation Agency)
project 316-2014  
``Capacity building in Mathematics and its applications'' 
under the SIDA bilateral program with the Makerere University 2015--2020.
\end{ack}

\section{Quasi(sub/super)minimizers}
\label{sect-qmin} 

Above we defined what quasiminimizers and quasisuperminimizers
are.
A function $u$ is a \emph{$Q$-quasisubminimizer}
if $-u$ is a $Q$-quasisuperminimizer.
Our definition of quasiminimizers (and quasisub- and quasisuperminimizers)
is one of several equivalent possibilities,
see Proposition~3.2 in A.~Bj\"orn~\cite{ABkellogg}.
In particular, 
we will use that
it is enough to require
\eqref{eqmain} to hold for all (nonnegative) $\phi \in \Lipc(\Om)$,
where $\Lipc(\Om)$ denotes the space of all Lipschitz functions
with compact support in $\Om$.

When $Q=1$ we usually drop ``quasi'' and say
(sub/super)minimizer.
Being a (sub/super)minimizer is the same
as being a (weak) (sub/super)solution of the 
\p-Laplace equation 
\[
    -\Div(|\nabla u|^{p-2} \nabla u)=0,
\]
see Chapter~5 in Heinonen--Kilpel\"ainen--Martio~\cite{HeKiMa}.
The function $u$ is a supersolution of this equation
if the left-hand side is nonnegative in a weak sense.

By Giaquinta--Giusti~\cite[Theorem~4.2]{GG2}, 
a $Q$-quasiminimizer can be
modified on a set of measure zero so that it becomes locally
H\"older continuous in $\Om$.
This continuous $Q$-quasiminimizer is called a
  \emph{$Q$-quasiharmonic} function,
and a \emph{\p-harmonic} function is a continuous minimizer.

If $u$ is a quasisuperminimizer,
then  $au+b$ is also a quasisuperminimizer whenever $a \ge 0$ and $b \in \R$.
Also, $u$ is a $Q$-quasiminimizer if and only if 
it is both a $Q$-quasisubminimizer and a $Q$-quasisuperminimizer.
Quasisuperminimizers are invariant under scaling 
in the following way.

\begin{lem}\label{lem1}
Let $\Om \subset \R^n$ be open
and $\tau>0$.
If $u: \Omega \rightarrow \R^n$  is  a $Q$-quasi\-super\-minimizer in 
$\Omega$,
then $v(x) := u(\tau x)$
is a $Q$-quasisuperminimizer in 
\[
   \Omega_{\tau} := \{ x: \tau x \in \Omega\}.
\]
\end{lem}

\begin{proof}
Let $\phi \in  W_{0}^{1,p}(\Omega_\tau)$ be nonnegative.
Then $\phit(x):=\phi(x/\tau) \in W_{0}^{1,p}(\Omega)$.
Hence, as $u: \Omega \rightarrow \R^n$  is  a $Q$-quasisuperminimizer in 
$\Omega$,
\begin{align*} 
   \int_{\varphi \neq 0} |\nabla v|^{p}\, dx
   & = \tau^{p-n} \int_{\phit \neq 0} |\nabla u|^{p}\, dx \\
   & \le  Q \tau^{p-n}\int_{\phit \neq 0} |\nabla( u + \phit)|^{p}\, dx
   =   Q \int_{\phi} |\nabla(v + \phi)|^{p}\, dx.
  \end{align*}
Thus $v$ is a $Q$-quasisuperminimizer in $\Omega_{\tau}$.
\end{proof}

The following definition will play a role in Section~\ref{sect-min}.
The finiteness requirement is important for this to be useful,
and is always fulfilled when $\itoverline{G}$ is a compact
subset of $\Om$, by the definition of quasiminimizers.

\begin{deff} \label{deff-determining}
If $u$ is a $Q$-quasiminimizer in $\Om \subset \R^n$ we say that
$u$ has the \emph{maximal \p-energy allowed by $Q$}
on the open set  $G\subset\Om$ if
\[
       \int_G |\nabla u|^{p}\, dx 
    = Q \int_G |\nabla v|^{p} \, dx < \infty,
\]
where $v$ is the minimizer in $G$ with boundary values
$v=u$ on $\bdy G$.
\end{deff}

For further discussion
on quasi(super)minimizers, as well as references to the literature,
we refer to 
Bj\"orn--Bj\"orn~\cite{BBpower}
and 
Bj\"orn--Bj\"orn--Korte~\cite{BBKorte}.
We will mainly be interested 
in radially symmetric functions on $\R^n$, $n \ge 2$.

\medskip

\emph{For the rest of this section,
    as well as in most of Section~\ref{sect-min},
we will only consider the case when $p=n$, the conformal dimension.}

\medskip

Define the annulus
\[
    A_{r_1,r_2}=\{x \in \R^n : r_1 < |x| < r_2\},
    \quad \text{where } 0 \le r_1 < r_2 \le \infty.
\]

In the conformal case ($p=n$) the logarithm
$\log |x|$ is an $n$-harmonic function,
and log-powers are quasiminimizers as we shall see.
These log-powers and their optimal quasisuperminimizing
constants  will be the crucial ingredients in Section~\ref{sect-min},
when investigating the increase in the quasisuperminimizing constant 
for the minimum of two quasisuperminimizers. 
The optimal quasiminimizing
and quasisub/superminimizing
constants for power functions
and log-powers $(-\log|x|)^\al$  on punctured unit balls were obtained in 
Bj\"orn--Bj\"orn~\cite{BBpower}.
These 
are rather easily shown to apply also
to the annuli $A_{\ga,1}$ with $0<\ga<1$, see
Theorem~\ref{thm-qmin-log-power-p=n} below.
We also need to consider log-powers $\log^\alp |x|:=(\log |x|)^\alp$
on the annuli $A_{1,\ga}$, $\ga>1$,
and their optimal quasiminimizing
and quasi\-sub/super\-minimizing
constants provided by the following
result.

\begin{thm} \label{thm-qmin-log-power-p=n-new}
Let $1 < \ga \le \infty$, $\alp>1-1/n$
and $u(x)=\log^\alp |x|$.
Then $u$ is a quasiminimizer in $A_{1,\ga}$
with 
\[ 
	\Qaln= \frac{\alp^n}{n\alp-n+1}
\] 
being the best quasiminimizing constant.

Moreover, $u$ is a $Q$-quasi\/{\rm(}sub\/{\rm/}super\/{\rm)}minimizer 
in\/ $A_{1,\ga}$ as given in Table~\ref{table-p=n-log},
where $Q$ in Table~\ref{table-p=n-log} is the best
quasi\/{\rm(}sub\/{\rm/}super\/{\rm)}\-mi\-ni\-mi\-zing constant.
Furthermore, 
$u$ has the \emph{maximal $n$-energy allowed by $\Qaln$}
on $A_{1,\ga}$ if $\ga <\infty$.
\end{thm}

\begin{table}[t]
\begin{center}
\begin{tabular}{|c||c|c|c|} 
\hline
& quasi- & quasi- & quasi- \\
 & minimizer & subminimizer &  superminimizer  \\
\hline \hline
$\displaystyle 1-\frac{1\strutheight}{n\strutdepth}<\al<1$ 
& $Q=\Qaln$ & $Q=\Qaln$ & $Q=1$   \\ \hline
$\displaystyle \al=1 \strutheight \strutdepth$
& $Q=1$ & $Q=1$  & $Q=1$   \\ \hline
$\displaystyle \al>1 \strutheight \strutdepth$
& $Q=\Qaln$ & $Q=1$ & $Q=\Qaln$   \\ \hline
\end{tabular}
\end{center}
\caption{Optimal quasiminimizing and
  quasisub/superminimizing constants of 
$\log^\alp |x|$ in $A_{1,\ga}$
and 
$(-\log |x|)^\al$ in $A_{\ga,1}$,
  as provided by Theorems~\ref{thm-qmin-log-power-p=n-new}
  and~\ref{thm-qmin-log-power-p=n}.
}
\label{table-p=n-log}
\end{table}

The proof is a modification of the proofs 
of Theorems~7.3 and~7.4 in \cite{BBpower}.
For the reader's convenience we provide the details.

\begin{proof}
Let $\phi(r)=\log^\alp r$,
$1<r_1<r_2<\ga$, $G=(r_1,r_2)$,
$s_1=\log r_1$, $s_2=\log r_2$ and $S=s_2/s_1>1$.

The $n$-energy of $u$ in $A_{r_1,r_2}$
is given by 
\[ 
I_u(A_{r_1,r_2})  \coloneqq \int _{\Omega}|\nabla u|^{n} \,dx
  = c_{n -1}\int_{r_1}^{r_2}|\varphi^{\prime}(r)|^{n} r^{n - 1} \,dr
   =: c_{n-1}\hat{I}_{\varphi}(G), 
\]
where $c_{n -1}$ is the $(n-1)$-dimensional surface area on the sphere 
$\mathbf{S}^{n-1}$. 
Moreover
\begin{align}
\label{eq-Iphi}
   \Ih_\phi(G)
    &  = \int_{r_1}^{r_2} \alp^n \frac{(\log r)^{n\alp-n}}{r^n} r^{n-1}\, dr
      = \int_{s_1}^{s_2} \alp^n t^{n\alp-n} \,dt
\\
    &  = \Qaln     ( s_2^{n\alp-n+1} - s_1^{n\alp-n+1})
     = \Qaln s_1^{n\alp-n+1}(S^{n\alp-n+1}-1).
   \nonumber
\end{align}
A minimizer is given by $\psi(r)=\log r$, and we have
letting $\alp=1$ above,
\[
   \Ih_\psi(G) = s_1 (S-1).
\]
We want to compare the energy $\Ih_\phi(G)$ with the energy $\Ih_\eta(G)$
  of the minimizer $\eta=a\psi+b$ having the same boundary values on 
$\bdy G$ as $\phi$. 
As
\[
     a=\frac{s_2^\alp-s_1^\alp}{s_2-s_1}=s_1^{\alp-1} \frac{S^\alp-1}{S-1},
\]
their quotient is
\begin{equation}    \label{eq-quotient-k(S)}
k(S)   :=   \frac{\Ih_\phi(G)}{\Ih_\eta(G)}
=   \frac{\Ih_\phi(G)}{|a|^n\Ih_\psi(G)}
=  \Qaln 
    \frac{S^{n\alp-n+1}-1}{S-1}   \biggl(\frac{S-1}{S^\alp -1}\biggr)^n,
\end{equation}
which only depends on $S$.

Let $s=\sqrt{s_1 s_2}$ and let $\eta_1$ and $\eta_2$
be the minimizers of the $\Ih$-energy 
on $G_1=(e^{s_1},e^s)$ resp.\  $G_2=(e^s,e^{s_2})$
having the same boundary values as $\phi$ on $\bdy G_1$ resp.\ $\bdy G_2$.
Also let $\etat=\eta_1 \chi_{G_1} + \eta_2 \chi_{G \setm G_1}$.
Then, as $s/s_1=s_2/s=\sqrt{S}$, we see that
\begin{align*}
    \Ih_\phi(G)
      &= k(S) \Ih_\eta(G) 
      \le k(S) \Ih_\etat(G) 
      = k(S) (\Ih_{\eta_1}(G_1)+\Ih_{\eta_2}(G_2)) \\
      &=  k(S) \biggl(\frac{\Ih_{\phi}(G_1)}{k(\sqrt{S})}+
           \frac{\Ih_{\phi}(G_2)}{k(\sqrt{S})} \biggr) 
      =  \frac{k(S)}{k(\sqrt{S})} \Ih_{\phi}(G).
\end{align*}
As $0 <\Ih_{\phi}(G) < \infty$, we find that $k(S) \ge k(\sqrt{S})$,
and thus
\begin{equation} \label{eq-sup}
   \sup_{S> 1} k(S)=\lim_{S \to \infty} k(S)= \Qaln.
\end{equation}
Comparing $u$ with $x \mapsto \eta(|x|)$ shows that
the quasiminimizing constant for $u$ cannot be less than $\Qaln$.

To show that $\Qaln$ will do, let $\omega$ be a function such 
that $\omega - \phi \in \Lipc((1,\ga))$.
The open set 
\[
 V=\{x \in (1,\ga) : \omega(x) \ne \phi(x)\}
\]
 can be written as a countable (or finite)  union
of pairwise disjoint intervals $\{I_j\}_{j}$.
We find from \eqref{eq-quotient-k(S)} and~\eqref{eq-sup} that
\begin{equation}
     \Ih_\phi(V) = \sum_j \Ih_\phi(I_j)
          \le \sum_j \Qaln\Ih_\omega(I_j)
          = \Qaln\Ih_\omega(V).
\label{eq-Ih-phi-le-Ih-om}
\end{equation}
Hence $\phi$ is indeed a $\Qaln$-quasiminimizer for the energy
$\Ih$ on $(1,\ga)$.

Next, we turn to $u$.
Let $v$ be such that
$v-u \in \Lipc(A_{1,\ga})$.
Also let 
\[
    \Om=\{x \in A_{1,\ga}: v(x) \ne u(x)\}.
\]
Using polar coordinates $x=(r,\theta)$,
where $r>1$ and $\theta  \in \Sphere^{n-1}$,
let 
\[
V_\theta=\{r : (r,\theta) \in \Om\}
\quad \text{and} \quad
v_\theta(r)=v(r,\theta).
\]
We then find, applying \eqref{eq-Ih-phi-le-Ih-om} to $G=V_\theta$, that
\begin{align*}
    I_u(\Om) &= \int_{\Sphere^{n-1}} \Ih_\phi(V_\theta) \, d\theta
             \le \int_{\Sphere^{n-1}} \Qaln\Ih_{v_\theta}(V_\theta) \, d\theta \\
             &=  \Qaln\int_{\Om} \biggl|\frac{\bdy v}{\bdy r}\biggr|^n  \, dx
             \le \Qaln  \int_{\Om} |\grad v|^n  \, dx
             = \Qaln I_v(\Om),
\end{align*}
showing that $u$ is indeed a $\Qaln$-quasiminimizer in $A_{1,\ga}$.

It follows directly that the constants in the quasiminimizer column in
Table~\ref{table-p=n-log} are correct.
By Lemma~\ref{lem-sub-super-log} below,
$u$ is a subminimizer if $\alp \ge 1$ and
a superminimizer if $1-1/n < \alp  \le 1$.
As $u$ is a $Q$-quasiminimizer if and only if 
it is both a $Q$-quasisubminimizer and a $Q$-quasisuperminimizer,
it follows that $\Qaln$ is 
the optimal quasisubminimizing constant when $1-1/n < \alp  \le 1$,
and the optimal quasisuperminimizing constant when $\alp  \ge 1$.

Finally, if $\ga < \infty$, then  it follows from \eqref{eq-Iphi}
that 
\[
     \Ih_\phi(A_{1,\ga})
      = \Qaln   (\log \ga)^{n\alp-n+1}
\]
and that the minimizer with the same boundary values
has $n$-energy $(\log \ga)^{n\alp-n+1}$, i.e.\
$u$ has the \emph{maximal $n$-energy allowed by $\Qaln$}
on $A_{1,\ga}$.
\end{proof}

\begin{lem} \label{lem-sub-super-log}
Let $u(x)=\log^\alp |x|$. 
Then $u$ is a superminimizer  in $A_{1,\infty}$ 
if and only if\/ $0 \le \alp \le 1$.
Similarly,   $u$ is a subminimizer  in $A_{1,\infty}$
if and only if $\alp \le0$ or $\alp \ge 1$.
\end{lem}

\begin{proof}
A straightforward calculation shows that
\[
    -\Div(|\nabla u(x)|^{n-2}\nabla u(x))
    = \alp (1-\alp) |\alp|^{n-2}  (n-1)
	(\log |x|)^{(n-1)\alp -n} |x|^{-n}
\]
for $x\in A_{1,\infty}$.
The sign of this expression is the same as of $\alp(1-\alp)$.
The function $u$ is, by definition,
a superminimizer if this expression is nonnegative, and
a subminimizer if it is nonpositive throughout $A_{1,\infty}$,
which thus happens exactly as stated.
\end{proof}

We also need the following result, which
is essentially from Bj\"orn--Bj\"orn~\cite{BBpower}.

\begin{thm} \label{thm-qmin-log-power-p=n}
\textup{(\cite[Theorems~7.3 and~7.4]{BBpower})}
Let $0 \le \ga < 1$, $\alp>1-1/n$
and $u(x)=(-\log |x|)^\alp$.
Then $u$ is a quasiminimizer in $A_{1,\ga}$
with 
\[ 
	\Qaln= \frac{\alp^n}{n\alp-n+1}
\] 
being the best quasiminimizing constant.

Furthermore, $u$ is a $Q$-quasi\/{\rm(}sub\/{\rm/}super\/{\rm)}minimizer 
in\/ $A_{\ga,1}$ as given in Table~\ref{table-p=n-log},
where $Q$ in Table~\ref{table-p=n-log} is the best
quasi\/{\rm(}sub\/{\rm/}super\/{\rm)}\-mi\-ni\-mi\-zing constant.
Also, 
$u$ has the \emph{maximal $n$-energy allowed by $\Qaln$}
on $A_{\ga,1}$ if $0 < \ga <1$.

\end{thm}

\begin{proof}
When $\ga=0$, the first part 
follows from Theorem~7.3 in \cite{BBpower}.
The proof therein holds equally well when $\ga>0$
as $S$ therein still ranges over $1<S<\infty$,
cf.\ the proof of Theorem~\ref{thm-qmin-log-power-p=n-new} above.

Similarly, the second part follows from Theorem~7.4 in \cite{BBpower},
where again the arguments hold also when $\ga >0$.
Finally, the last part follows from the formula at the bottom
of p.~314 in \cite{BBpower}, cf.\ the end of 
the proof of Theorem~\ref{thm-qmin-log-power-p=n-new} above.
\end{proof}

\section{The increase in the 
quasisuperminimizing constant}
\label{sect-min}

In this section we are going to use the log-powers
considered in Section~\ref{sect-qmin} to construct
higher-dimensional analogues of the examples
in Bj\"orn--Bj\"orn--Korte~\cite[Section~3]{BBKorte},
concerning the optimality of
\eqref{eq-def-q-for-min} in Theorem~\ref{thm-BBK}.

Fix $n \ge 2$.
We 
study quasisuperminimizers
on the annulus 
\[ 
   A:=A_{1/e,1}=\{x \in \R^n : 1/e < |x|<1\}.
\]
  
As before, we let $p=n$ be the conformal dimension.

\begin{example} \label{ex-log-powers}
Given $Q>1$, there are exactly two exponents $1-1/n<\al'<1<\al$
such that $Q=Q_{\al,n}=Q_{\al',n}$, where 
\begin{equation} \label{eq-def-Q-al}
	\Qaln= \frac{\alp^n}{n\alp-n+1}.
\end{equation}
This is easily shown by differentiating~\eqref{eq-def-Q-al} 
with respect to $\al$ and noting
that the derivative is negative for $\al<1$ and positive for $\al>1$,
and that $Q_\al\to\infty$ as $\al\to 1-1/n$ and as $\al\to\infty$.

We let 
\begin{equation}    \label{eq-def-uQ-ubQ}
u_Q(x)=\log^\al|ex| \quad \text{and} \quad \ub_Q(x)=1-(-\log|x|)^{\al'}.
\end{equation}
Then $u_Q(y)=\ub_Q(y)=0$ if $|y|=1/e$, and $u_Q(y)=\ub_Q(y)=1$ if $|y|=1$.
By Theorem~\ref{thm-qmin-log-power-p=n-new} and
Lemma~\ref{lem1}, $u_Q$ is a subminimizer and a $Q$-quasisuperminimizer 
in $A$.
The same is true for $\ub_Q$ by Theorem~\ref{thm-qmin-log-power-p=n}.
\end{example}

It follows from 
Theorems~\ref{thm-qmin-log-power-p=n-new}
and~\ref{thm-qmin-log-power-p=n}
that $u_Q$ has the maximal $n$-energy allowed by $Q$ on each 
annulus $A_{1/e,\ga}$, $1/e < \ga \le 1$, 
while $\ub_Q$ has the maximal $n$-energy 
allowed by $Q$ on each annulus $A_{\ga,1}$, $1/e \le  \ga  <1$.
This will be of crucial importance when proving
  Theorem~\ref{thm0}, which we will do now.

\begin{proof}[Proof of Theorem~\ref{thm0}]
Let $1-1/n < \al_2<1 < \al_1$
be such that $Q_1=Q_{\al_1,n}$ and  $Q_2=Q_{\al_2,n}$ as in \eqref{eq-def-Q-al},
and let $u_1:=u_{Q_1}$ and $u_2:=\ub_{Q_2}$ be the corresponding quasiminimizers
as in \eqref{eq-def-uQ-ubQ}.
By Example~\ref{ex-log-powers},
$u_j$ is a 
subminimizer and a $Q_j$-quasisuperminimizer in $A$, $j=1,2$.

By Theorem~\ref{thm-BBK}, the
function $u:= \min\{u_1, u_2\}$ is a $\bQ$-quasisuperminimizer in
$A$ with the quasisuperminimizing constant 
$\bQ$ given by \eqref{eq-def-q-for-min}.
Let $v=\log |ex|$, which is the minimizer on $A$
with the same boundary values as $u$, $u_1$ and $u_2$.
As $u_1$ and $u_2$ are subminimizers they are less than $v$, which
can also be seen directly.
Thus $u < v$ in $A$.

We are going to show that $u$ is not a
$Q_2$-quasisuperminimizer on $A$.
As $u <v$, 
to do this it suffices to show that the $n$-energy
\[
 I_u:= \int_{A} |\nabla u |^n \, dx 
  > Q_2I_v,
\]
where
\begin{equation} \label{eq-Iv}
    I_v= \int_{A} |\nabla v|^n \, dx 
      = c_{n-1} \int_{1/e}^1 r^{-n} r^{n-1}\, dr
      = c_{n-1}.
\end{equation}

For convenience, write $u_j(r)=u_j(x)$ when $r=|x|$.
There is a unique number $r_0 \in (1/e, 1)$
such that $u_1(r_0)= u_2(r_0)$ (see below), i.e.\ such that
\begin{equation}\label{eqlem1}
          1 = (-\log r_0)^{\alpha_2} +  \log^{\alpha_1} er_0.
\end{equation}

To see that there is a unique solution,  we consider $w = u_2 -u_1$
and note that $w(1/e) = w(1) =0$. Since
$w^{\prime}(1/e)> 0$ and $w^{\prime}(1)=\infty$, there is at
least one $r\in (1/e, 1)$ such that $w(r) = 0$. 
Moreover,
$w^{\prime}(r) = 0$ if and only if
\[
f(r) \coloneqq \log (er)(-\log r)^{\beta} =
\biggl(\frac{\alpha_1}{\alpha_2}\biggr)^{1/(1 -\alpha_1)} >0,
\quad \text{where } \beta = \frac{1 - \alpha_2}{\alpha_1 -  1}> 0.
\]
Note that $f'(r) > 0$ if and only if 
$0 < r < e^{-\beta/(\beta + 1)}$  
and that $f(r)$ attains its maximum at
(and only at) $r = e^{-\beta/(\beta + 1)}$. 
Since
$f(1/e) = f(1) = 0$, there are
at most two solutions to $w^{\prime}(r) =0$, and thus there can be
at most one solution to \eqref{eqlem1} which must lie in between the
two local extrema of $w$.

Since $u_2$ is a subminimizer in
$A$ we see that 
\[
 \int_{A_{1/e,r_0}}|\nabla u_1|^n \, dx 
   >  \int_{A_{1/e,r_0}}|\nabla u_2|^n \, dx,
\]
where the strict inequality follows from the uniqueness of solutions to
obstacle problems 
(see e.g.\ Bj\"orn--Bj\"orn~\cite[Theorem~7.2]{BBbook})
and from the fact that $u_1$ and $u_2$ differ on a
set of positive measure. 
Hence 
\[ 
 \int_{A}|\nabla u|^n \, dx 
=  \int_{A_{1/e,r_0}}|\nabla u_1|^n \, dx 
   +  \int_{A_{r_0,1}}|\nabla u_2|^n \, dx
   >  \int_{A}|\nabla u_2|^n \, dx
= Q_2 I_v, 
\]
where the last equality follows from the fact that 
$u_2$ 
has the maximal $n$-energy allowed by $Q_2$ on $A$.
As $Q_2 \ge Q_1$ this concludes the proof.
\end{proof}

Theorem~\ref{thm0} shows that in general there is an increase in
the quasisuperminimizing constant when taking the minimum of two
quasiminimizers  but does not give any quantitative estimate of the
increase. 
Next, we are going to analyse the construction
in the proof of Theorem~\ref{thm0} more carefully to
get explicit lower bounds for the increase 
in the quasisuperminimizing constant.

Given $Q_1, Q_2 > 1$ and $p=n \ge 2$,  
 let  $1 - 1/n < \alpha_2 < 1 < \alpha_1$ 
be such that $Q_1 = Q_{\alpha_1,n}$ and $Q_2 = Q_{\alpha_2,n}$ 
as in \eqref{eq-def-Q-al},
and let $u_1:=u_{Q_1}$ and $u_2:=\ub_{Q_2}$ be the corresponding quasiminimizers
as in \eqref{eq-def-uQ-ubQ}.
Let $r_0 \in (1/e, 1)$ be as in \eqref{eqlem1}. 
Contrary to Theorem~\ref{thm0} we here allow for $Q_1 > Q_2$
(the assumption $Q_1 \le Q_2$ in Theorem~\ref{thm0} is only used
  at the very end of its proof).

It follows from Theorems~\ref{thm-qmin-log-power-p=n}
and~\ref{thm-qmin-log-power-p=n-new} that
$u_1=u_{Q_1}$ 
has the maximal $n$-energy allowed by $Q_1$ on 
$A_{1/e,r_0}$, 
while $u_2=\ub_{Q_2}$ has the maximal $n$-energy 
allowed by $Q_2$ on $A_{r_0,1}$.
Using this, we can calculate the $n$-energy of
$u=\min\{ u_1, u_2\}$
as
\begin{align}
 \int_{A}|\nabla u|^n \, dx 
&=  \int_{A_{1/e,r_0}}|\nabla u_1|^n \, dx 
   +  \int_{A_{r_0,1}}|\nabla u_2|^n \, dx \nonumber \\
&=   c_{n-1} \bigl(   Q_1(\log er_0)^{n(\alpha_1 -1) + 1} 
    + Q_2(-\log r_0)^{n(\alpha_2 -1) + 1}\bigr) \nonumber \\
&     =:c_{n-1}\Qhat. \label{eq-Qhat}
\end{align}

Comparing this value with the energy $I_v=c_{n-1}$, given by \eqref{eq-Iv},
of the minimizer
$v$ with the same boundary values as $u$ on $A$
we see that $u$ is not a $Q$-quasisuperminimizer
for any $Q < \Qhat$.

\begin{table}[t]
  \begin{center}
    \begin{tabular}{|c|c|c|c|c|c|clllllll}
\hline
     $Q$ &$ p=n=2$ & $p = n=3$&$p=n=10$&$p=n=100$&
$\bQ=\frac{\displaystyle 2Q\strutheight}{\displaystyle Q+1\strutdepth}$ \\
\hline
$1.001$ &1.001480660&1.001480663&1.001480664&1.001480665&1.001500250\\
\hline
1.01&1.014825154&1.014825447&1.014825583&1.014825593&1.015024876\\
\hline
1.125&1.188100103&1.188143910&1.188164386&1.188165836&1.191176471\\
\hline
     2  &2.619135721 &2.621145314&2.622093879&2.622161265&2.666666667\\ 
      \hline
     10 &17.67321156& 17.70495731&17.72058231 &17.72170691&18.18181818\\ 
\hline
   100  &196.3948537&196.5222958&196.5905036&196.5955633& 198.0198020\\ 
\hline
    \end{tabular}
\caption{$\Qhat$ for certain values of $p=n$ with $Q_1=Q_2=Q$,  as well as 
 $\protect\bQ$ from Theorem~\ref{thm-BBK}.
  }
\label{table2}
  \end{center}
\end{table}

For specific values of $Q_1$, $Q_2$ and $p=n$, one
can calculate $\Qhat$ numerically,
(after first calculating
$\alp_1$, $\alp_2$ and $r_0$ numerically),
which we have done using Maple 18.
These results are presented in 
Table~\ref{table2}, which  shows the values
of $\Qhat$ for certain values of $p=n$ and with $Q_1=Q_2=Q$.

\begin{remark} \label{remark1}
The figures for $\Qhat$
in the columns  for $p=2$ and $p=100$ 
in Table~\ref{table2} above
are identical to the corresponding
columns for $\Qt$ in Table~2 in \cite{BBKorte}
(there are no columns for $p=3$ and $p=10$ therein).
This suggests that the increase in 
quasisuperminimizing constant in
the example above is identical 
to the increase in the $1$-dimensional example
in \cite[pp.\ 271--272]{BBKorte}.
This is indeed true, as we shall now show,
not only when $Q_1=Q_2$.

Let as before $p=n \ge 2$ be an integer.
(In the example in \cite[pp.\ 271--272]{BBKorte},
the underlying space is $\R$, but 
$p$ can be an arbitrary real number $>1$.
To compare it with our construction above we need $p$ 
to be an integer.)

Let as above $Q_1, Q_2>1$ be given, and choose
$1 - 1/n < \alpha_2 < 1 < \alpha_1$ 
such that $Q_1 = Q_{\al_1,n}$ and $Q_2 = Q_{\al_2,n}$ 
as in \eqref{eq-def-Q-al}.
With $p=n$ this choice of $\alpha_1$ and $\alpha_2$
also satisfies (3.2) in \cite{BBKorte}.
Next choose $r_0 \in (1/e,1)$ to be the unique
solution of \eqref{eqlem1}.
Then $\Qhat$ is given by \eqref{eq-Qhat}.

To relate this to $\Qt$ in \cite{BBKorte}, we
let $x_0=\log er_0 \in (0,1)$.
It then follows from \eqref{eqlem1} that
\[
        1=(1-x_0)^{\alp_2} + x_0^{\alp_1},
\]
i.e.\ $x_0$ is the unique solution of this equation,
which is the same as equation (3.4) in \cite{BBKorte}.
(That the solution is unique was shown in \cite{BBKorte},
but also follows from the uniqueness of $r_0 \in (1/e,1)$.)
Using (3.6) in \cite{BBKorte} (with $p=n$) we see that 
\begin{align} \label{eq-Qt}
\Qt &= Q_1 x_0^{n(\al_1-1)+1} + Q_2(1-x_0)^{n(\al_2-1)+1} \\
    & = Q_1 (\log er_0)^{n(\al_1-1)+1} + Q_2(-\log r_0)^{n(\al_2-1)+1}
    = \Qhat. \label{eq-Qt-Qhat}
\end{align}
\end{remark}

\begin{table}[t]
  \begin{center}
    \begin{tabular}{|r|r|r|r|r|r|r|}
\hline
\multicolumn{1}{|c|}{$Q_1$} & 
\multicolumn{1}{|c|}{$Q_2$} & 
\multicolumn{1}{|c|}{$p=1.2$} & 
\multicolumn{1}{|c|}{$p=2$} & 
\multicolumn{1}{|c|}{$p=10$} & 
\multicolumn{1}{|c|}{$p=100\strutdepth$} & 
\multicolumn{1}{|c|}{$\bQ\strutheight$}  \\
\hline
\stab{2}{10}  & \stab{10}{2}  
& \stab{10.450759}{10.222890} 
& \stab{10.474426}{10.293133} 
& \stab{10.477869}{10.309651} 
& \stab{10.477946}{10.310050} 
& 10.526316  \\ \hline
\stab{9}{10}  & \stab{10}{9} 
& \stab{16.513457}{16.473657}
& \stab{16.719374}{16.689656}
& \stab{16.762792}{16.736154}
& \stab{16.763819}{16.737258}
& 17.191011 \\ \hline
\stab{2}{100}  & \stab{100}{2} 
& \stab{100.427051}{100.055345}
& \stab{100.450836}{100.111528}
& \stab{100.454265}{100.134063}
& \stab{100.454342}{100.132692}
& 100.502513\\ \hline
\stab{10}{100}  & \stab{100}{10} 
& \stab{107.287586}{106.251592}
& \stab{107.542028}{106.758915}
& \stab{107.596390}{106.910025}
& \stab{107.597679}{106.913964}
& 108.108108\\ \hline
\stab{90}{100}  & \stab{100}{90} 
& \stab{185.787954}{185.723660}
& \stab{186.446301}{186.399453}
& \stab{186.634352}{186.594523}
& \stab{186.639194}{186.599568}
& 188.020891\\ \hline
    \end{tabular}
\caption{$\Qt$ (given by \eqref{eq-Qt})
for certain values of $p$, $Q_1$ and $Q_2$,   as well as 
 $\protect\bQ$ from Theorem~\ref{thm-BBK}.
When $p$ is an integer these are also the
values of $\Qhat$ by \eqref{eq-Qt-Qhat}.
}

\label{table3}
  \end{center}
\end{table}

%
%

\begin{remark}
The function $\Qt$ depends on $Q_1$, $Q_2$ and
$p$, i.e.\ $\Qt=\Qt(Q_1,Q_2,p)$.
Given $1<Q_1<Q_2$ (and $p$) it is natural to ask
which is larger of $\Qt(Q_1,Q_2,p)$ and $\Qt(Q_2,Q_1,p)$.
We have calculated some values of $\Qt(Q_1,Q_2,p)$
using Maple 18,
see  Table~\ref{table3}.
They all indicate that
\begin{equation}  \label{eq-Qt>Qt}
\Qt(Q_1,Q_2,p) > \Qt(Q_2,Q_1,p)
\quad \text{if } 1<Q_1<Q_2.
\end{equation}
Due to the intricate formula \eqref{eq-Qt}
for $\Qt$, involving $x_0$, we have not been able to show
that this is always the case.

The formula for $\Qt$ is valid  also for nonintegers $p>1$,
and $p=1.2$ is included in Table~\ref{table3}.
However, if $p$ is an integer 
then $\Qhat=\Qt$, by \eqref{eq-Qt-Qhat},
and the same reasoning about the comparison in \eqref{eq-Qt>Qt}
applies to $\Qhat$.
\end{remark}

\section{Adding dimensions}
\label{sect-dummy}

One way of making higher-dimensional examples from
  lower-dimensional ones is to add dummy variables.
The tensor product $u_1 \otimes u_2(x,y)=u_1(x)u_2(y)$ and tensor sum
$u_1 \oplus u_2=u_1(x)+u_2(y)$ of two harmonic functions
is again harmonic, a fact that
is well known and easy to prove.
The corresponding fact is false for \p-harmonic functions,
but it was observed in Bj\"orn--Bj\"orn~\cite{BBtensor}
that the tensor product and sum are quasiminimizers.
They moreover showed that the tensor product and sum of
quasiminimizers are again quasiminimizers, but typically with an
increase in the quasiminimizing constant even if both are $1$.
However, if one of the quasiminimizers is constant then the increase
in the quasiminimizing constant can be avoided, a fact that we shall use.
We first recall the following consequence of
the results in \cite{BBtensor}.

\begin{thm} \label{thm-BBtensor}
  Let $u$ be a $Q$-quasisuperminimizer in $\Om\subset \R^n$
  and let $I \subset \R$ be an interval.
  Then $\ub=u \otimes 1$ 
  is a $Q$-quasisuperminimizer in $\Om \times I$.
\end{thm}

This result is true also if the first space $\R^n$ is equipped with a
so-called \p-admissible weight $w$, see \cite{BBtensor}.
In particular, by Theorem~3 in \cite{BBtensor}, $w \otimes 1$ is then a
\p-admissible weight on $\R^{n+1}$.

\begin{proof}
  This is a special case of Theorem~7 in~\cite{BBtensor},
  with $u_1=u$, $u_2 \equiv 1$, $Q_1=Q$ and $Q_2=0$.
  As mentioned in \cite[p.~5196]{BBtensor}, one is allowed to let $Q_2=0$
  if $u_2$ is a constant function.
\end{proof}

For the purposes in this paper,
this is not enough since, typically, the obtained $Q$ is not
the optimal quasisuperminimizing constant for $\ub$ even if
it is for $u$.
But if $I$ is in addition unbounded then $Q$ is optimal for $\ub$ if
it is for $u$, as we shall now show.

\begin{thm} \label{thm-newtensor}
  Let $u$ be a $Q$-quasisuperminimizer in $\Om\subset \R^n$,
  where $Q$ is the optimal quasisuperminimizing constant.
  Let $I \subset \R$ be an unbounded interval.
  Then $\ub=u \otimes 1$
  is a $Q$-quasisuperminimizer in $\Om \times I$,
  with $Q$ again being optimal.
\end{thm}

\begin{proof}
By Theorem~\ref{thm-BBtensor}, $\ub$ is a
$Q$-quasisuperminimizer,
so it is only the optimality of $Q$ that needs to be shown.
If $Q=1$ there is nothing to prove, so we can assume that $Q>1$.

Let $0<\eps<Q$.
Since $Q$ is optimal,
there is 
a nonnegative $\phi \in \Lipc(\Om)$ such that
\[
   \int_{\phi \neq 0} |\nabla u|^{p}\,dx
    >       (Q-\eps)\int_{\phi \neq 0} |\nabla( u + \phi)|^{p}\,dx,
\] 
see 
the beginning of Section~\ref{sect-qmin}.
As the integral on the left-hand side is positive, also
the integral on the right-hand side must be positive,
as otherwise $u$ would not be a quasisuperminimizer at all.

Next, let $m>0$ be given.
Since $I$ is unbounded we can find $a \in \R$ so that
$[a,a+m+2] \subset I$.
Assume without loss of generality that $a=0$,
and let
\[
\phib=\phi \otimes \phi_2,
\quad \text{where }
\phi_2(t)=\begin{cases}
     0, & \text{if } t  \le 0 \text{ or } t \ge m+2, \\
     t, & \text{if } 0 \le t  \le 1, \\
     1, & \text{if } 1  \le t \le m+1, \\
     m+2-t, & \text{if } m+1 \le t   \le m+2.
\end{cases}
\]
Then
\[
  \int_{\phib \neq 0} |\nabla \ub|^{p}\,dx\,dt
  \ge m    \int_{\phi \neq 0} |\nabla u|^{p}\,dx
  >       (Q-\eps)m\int_{\phi \neq 0} |\nabla( u + \phi)|^{p}\,dx,
\]
while
\[
  \int_{\phib \neq 0} |\nabla (\ub+\phib)|^{p}\,dx\,dt
  = 2  \int_{\{\phi \ne 0\} \times (0,1)} |\nabla (\ub+\phib)|^{p}\,dx\,dt
   +m\int_{\phi \neq 0} |\nabla( u + \phi)|^{p}\,dx.
\]
Letting $m \to \infty$ and then $\eps \to 0$
shows that $Q$ is optimal, since
the last integral is nonzero.
\end{proof}

It now follows directly from Theorem~\ref{thm-newtensor}
that we can add a dummy variable
to the examples constructed
 in Section~\ref{sect-min}
and in
Bj\"orn--Bj\"orn--Korte~\cite[Section~3]{BBKorte}.
As long as we  consider the dummy variable
taken over 
an unbounded interval,
we obtain a new example with the same increase
in the quasisuperminimizing constant.
This can be iterated so that we can add
an arbitrary (but finite) number of dummy variables.
This way we get higher-dimensional examples on unbounded sets.
However, it follows from the following lemma
that by taking Cartesian products
with large enough bounded intervals, 
one can obtain similar bounded examples with an increase 
in the quasisuperminimizing constant, which is 
arbitrarily close to the increase in Section~\ref{sect-min}
resp.\ \cite[Section~3]{BBKorte}.

\begin{lem}
  Let $\Om_1 \subset \Om_2 \subset \ldots \subset \Om=\bigcup_{j=1}^\infty \Om_j$
  be an increasing sequence of open subsets of $\R^n$.
  If $u$ is a $Q$-quasisuperminimizer in $\Om_j$ for every $j$,
  then it is also a $Q$-quasisuperminimizer in $\Om$.
\end{lem}

\begin{proof}
  As mentioned at the beginning of Section~\ref{sect-qmin}
  it is enough to test
  \eqref{eqmain} with nonnegative $\phi \in \Lipc(\Om)$.
  By compactness, $\phi \in \Lipc(\Om_j)$ for some $j$,
  and thus   \eqref{eqmain}
  holds for this particular $\phi$ as 
  $u$ is a $Q$-quasisuperminimizer in $\Om_j$.
\end{proof}


\begin{thebibliography}{99}

\bibitem{ABkellogg} \art{\auth{Bj\"orn}{A}}
        {A weak Kellogg property for quasiminimizers}
        {Comment. Math. Helv.} {81} {2006} {809--825}

\bibitem{BBpower} \art{ Bj\"orn, A. \AND Bj\"orn, J.} 
	{Power-type quasiminimizers}
	{Ann. Acad. Sci. Fenn. Math.} {36} {2011} {301--319}

\bibitem{BBbook} \book{Bj\"orn, A. \AND Bj\"orn, J.}
        {\it Nonlinear Potential Theory on Metric Spaces}
    {EMS Tracts in Mathematics {\bf 17},
        European Math. Soc., Z\"urich, 2011}

\bibitem{BBtensor}  \art{\auth{Bj\"orn}{A} \AND \auth{Bj\"orn}{J}}	
        {Tensor products and sums of \p-harmonic functions, quasiminimizers
          and \p-admissible weights}
        {Proc. Amer. Math. Soc.} {146} {2018} {5195--5203}
    
\bibitem{BBKorte} \art{Bj\"orn, A.,  Bj\"orn, J.\ \AND Korte, R.}
     {Minima of quasisuperminimizers}
        {Nonlinear Anal.}{155} {2017} {264--284}	

\bibitem{GG1} \art{Giaquinta, M. \AND Giusti, E.}
         {On the regularity of the minima of variational integrals}
         {Acta Math.} {148} {1982} {31--46}

\bibitem{GG2} \art{Giaquinta, M. \AND Giusti, E.}
         {Quasi-minima}
         {Ann. Inst. H. Poincar\'e Anal. Non Lin\'eaire} {1} {1984} {79--107}

\bibitem{HeKiMa} \book{Heinonen, J., Kilpel\"ainen, T.\ \AND Martio, O.}
        {Nonlinear Potential Theory of Degenerate Elliptic Equations}
        {2nd ed., Dover, Mineola, NY, 2006}	

\bibitem{KiMa03} \art{Kinnunen, J. \AND Martio, O.}
        {Potential theory of quasiminimizers}
        {Ann. Acad. Sci. Fenn. Math.} {28} {2003} {459--490}

\bibitem{MaSb} \art{Martio, O. \AND Sbordone, C.}
	{Quasiminimizers in one dimension: integrability of the derivative,
	inverse function and obstacle problems}
	{Ann. Mat. Pura Appl.} {186} {2007} {579--590}


\end{thebibliography}
\end{document}